\documentclass[12pt]{amsart}
\usepackage{amsfonts}
\usepackage{amssymb}
\usepackage{graphicx}
\usepackage{pgf,tikz}
\usetikzlibrary{arrows}
\usepackage[letterpaper, left=2.5cm, right=2.5cm, top=2.5cm,
bottom=2.5cm,dvips]{geometry}
\usepackage[T1]{fontenc} 
\usepackage[backref]{hyperref}
\hypersetup{
    colorlinks,
    linkcolor={blue!60!black},
    citecolor={green!60!black},
    urlcolor={red!60!black}
}

\newtheorem{theorem}{Theorem}
\theoremstyle{plain}

\newtheorem{case}{Case}

\newtheorem{conjecture}{Conjecture}
\newtheorem{corollary}{Corollary}

\newtheorem{lemma}{Lemma}

\newtheorem{remark*}{Remark}

\newtheorem{case2}{Subcase}
\newtheorem{fact*}{Fact}

\DeclareMathOperator{\col}{col}

\DeclareMathOperator{\A}{A}

\numberwithin{equation}{section}

\begin{document}
\title[Graph polynomials and group coloring of graphs]{Graph polynomials and group coloring of graphs}

\author[B.~Bosek]{Bart\l omiej Bosek}
\address{Institute of Theoretical Computer Science, Faculty of Mathematics and Computer Science, Jagiellonian University, Krak{\'o}w, Poland}
\email{bartlomiej.bosek@uj.edu.pl}

\author[J.~Grytczuk]{Jaros\l aw Grytczuk}
\address{Faculty of Mathematics and Information Science, Warsaw University of Technology, Warsaw, Poland}
\email{j.grytczuk@mini.pw.edu.pl}

\author[G.~Gutowski]{Grzegorz Gutowski}
\address{Institute of Theoretical Computer Science, Faculty of Mathematics and Computer Science, Jagiellonian University, Krak{\'o}w, Poland}
\email{grzegorz.gutowski@uj.edu.pl}

\author[O.~Serra]{Oriol Serra}
\address{Department of Mathematics, Universitat Polit\`{e}cnica de Catalunya, Barcelona, Spain}
\email{oriol.serra@upc.edu}

\author[M.~Zaj\k{a}c]{Mariusz Zaj\k{a}c}
\address{Faculty of Mathematics and Information Science, Warsaw University of Technology, Warsaw, Poland}
\email{m.zajac@mini.pw.edu.pl}

\thanks{Supported by the Polish National Science Center, Grant Number: NCN 2019/35/B/ST6/02472. Oriol Serra acknowledges financial support from the Spanish Agencia Estatal de Investigaci\'on under project MTM2017-82166-P}

\begin{abstract}Let $\Gamma$ be an Abelian group and let $G$ be a simple graph. We say that $G$ is \emph{$\Gamma$-colorable} if for some fixed orientation of $G$ and every edge labeling $\ell:E(G)\rightarrow \Gamma$, there exists a vertex coloring $c$ by the elements of $\Gamma$ such that $c(y)-c(x)\neq \ell(e)$, for every edge $e=xy$ (oriented from $x$ to $y$).
    
Langhede and Thomassen proved recently that every planar graph on $n$ vertices has at least $2^{n/9}$ different $\mathbb{Z}_5$-colorings. By using a different approach based on graph polynomials, we extend this result to $K_5$-minor-free graphs in the more general setting of \emph{field coloring}. More specifically, we prove that every such graph on $n$ vertices is $\mathbb{F}$-$5$-choosable, whenever $\mathbb{F}$ is an arbitrary field with at least $5$ elements. Moreover, the number of colorings (for every list assignment) is at least $5^{n/4}$.
\end{abstract}
\maketitle

\section{Introduction}Let $\Gamma$ be an Abelian group and let $G$ be a simple graph. We say that $G$ is \emph{$\Gamma$-colorable} if for some orientation of $G$ and every edge labeling $\ell$ by the elements of $ \Gamma$ there exists a vertex coloring $c$ by the elements of $\Gamma$ such that $c(y)-c(x)\neq \ell(e)$, for every edge $e=xy$ (oriented from $x$ towards $y$). This notion was introduced by Jaeger, Linial, Payane and Tarsi~\cite{JaegerLPT} as a dual concept to \emph{group connectivity}.

 Answering a question posed in~\cite{JaegerLPT}, Lai and Zhang~\cite{LaiZhang1} proved that every planar graph is $\mathbb{Z}_5$-colorable. Recently, Langhede and Thomassen~\cite{LanghedeThomassen2} strengthened this result by proving that the number of $\mathbb{Z}_5$-colorings of every planar graph on $n$ vertices (for any fixed edge labeling) is at least $2^{n/9}$. The proof is elementary but quite involved.

In this paper we further extend these results by using the polynomial method. It is convenient to introduce a slightly more general setting. Let $\mathbb{F}$ be an arbitrary field and let $G$ be a simple graph. Suppose that each edge $e=xy$ of $G$ is assigned a triple $(a_e,b_e,c_e)\in \mathbb{F}^3$, with $a_e,b_e\neq0$. We say that $G$ is \emph{$\mathbb{F}$-colorable} if for every such edge labeling there exists a vertex coloring $f$ by the elements of $\mathbb{F}$ such that $a_ef(x)+b_ef(y)+c_e\neq 0$, for every edge $e=xy$. Clearly, $\mathbb{F}$-colorability of a graph implies its $\Gamma$-colorability, where $\Gamma$ is the additive group of the field $\mathbb{F}$. 

 Define a graph $G$ to be \emph{$\mathbb{F}$-$k$-choosable} if it is $\mathbb{F}$-colorable from arbitrary \emph{lists} of elements of $\mathbb{F}$, each of size $k$, assigned to the vertices.

\begin{theorem}\label{Theorem K_5 Minor Main} Let $G$ be a graph on $n$ vertices without a minor of $K_5$. Let $\mathbb{F}$ be an arbitrary field with at least $5$ elements. Then $G$ is $\mathbb{F}$-$5$-choosable. Moreover, the number of colorings (for any fixed list assignment and any fixed edge labeling) is at least $5^{n/4}$. 
\end{theorem}

The proof is based on the method of graph polynomials. We use two tools, the Combinatorial Nullstellensatz of Alon~\cite{AlonCN} and a result of Alon and F\"{u}redi~\cite{AlonFuredi} concerning the number of non-zero values of a polynomial evaluated at all points of a multidimensional grid.

Notice that Theorem~\ref{Theorem K_5 Minor Main} only covers Abelian groups that are additive groups of a field. For a general Abelian group $\Gamma$ of order at least $5$, Chuang, Lai, Omidi, Wang, and Zakeri proved in~\cite{ChuangLOWZ} by elementary methods that every $K_5$-minor free graph is $\Gamma$-5-choosable. Theorem~\ref{Theorem K_5 Minor Main} does not extend this result, but in the overlapping cases guarantees a stronger conclusion, and also easily implies it in the case of arbitrary cyclic groups $\Gamma$ (Theorem~\ref{Theorem Cyclic Groups}).

\section{The results}

\subsection{Graph polynomials}
Let $G$ be a simple graph on the set of vertices $V(G)=\{x_1,x_2,\dots,x_n\}$. Let $P_G$ be the \emph{graph polynomial} of $G$, defined by
\begin{equation}\label{Graph Polynomial}
P_G(x_1,x_2,\dots,x_n)=\prod_{x_ix_j\in E(G), i<j}(x_i-x_j).
\end{equation}
We identify symbols denoting vertices of $G$ with variables of $P_G$. We may consider $P_G$ as a polynomial over an arbitrary field $\Bbb F$. %Notice that $P_G$ is defined uniquely up to the sign, depending on the choice between the two possible expressions, $(x_i-x_j)$ or $(x_j-x_i)$, representing the edge $x_ix_j$ in the product (\ref{Graph Polynomial}). We shall assume that this choice is somehow fixed (e.g. by ordering the terms in each factor by their subscripts).

In the process of expanding the polynomial $P_G$, one creates monomials by picking one variable from each factor $(x_i-x_j)$. Thus, every monomial corresponds to the unique orientation of $G$ obtained by directing the edge $x_ix_j$ \emph{towards} the picked variable. Thus, the degrees of the variables in the monomial coincide with the in-degrees of the vertices in the corresponding orientation.

Let $\mathcal{M}_G$ denote the multi-set of all monomials arising in this way. So, the cardinality of $\mathcal{M}_G$ is equal to $2^m$, where $m=|E(G)|$, and the multiplicity of each monomial $M$ is equal to the number of orientations of $G$ sharing the same in-degree sequence (corresponding to the degrees of the variables in $M$). The \emph{sign} of a monomial $M\in \mathcal{M}_G$ is the product of signs of all variables picked to form $M$. The \emph{coefficient} of a monomial $M$ in $P_G$, denoted as $c_M(P_G)$, is the sum of signs of all copies of $M$ in $\mathcal{M}_G$. A monomial $M$ is called \emph{non-vanishing} in $P_G$ if $c_M(P_G)\neq 0$.

Suppose now that each edge $e=x_ix_j$ of a graph $G$ (oriented so that $i<j$) is assigned an arbitrary pair $(a_e,b_e)$ of non-zero elements of $\mathbb{F}$. We say that the edges of $G$ are \emph{decorated} with pairs $(a,b)$, and we define the corresponding \emph{decorated} graph polynomial $D_G$ in which every factor $(x_i-x_j)$ corresponding to the edge $e=x_ix_j$ decorated with $(a_e,b_e)$ is substituted with $(a_ex_i+b_ex_j)$:

\begin{equation}\label{Graph Polynomial Decorated}
D_G(x_1,x_2,\dots,x_n)=\prod_{e=x_ix_j\in E(G), i<j}(a_ex_i+b_ex_j).
\end{equation}
 
 Of course, different decorations may give different polynomials, but we denote the whole family of them with the same symbol $D_G$, hoping that this ambiguity will not cause too much confusion. 
 
\subsection{Combinatorial Nullstellensatz}

For a monomial $M$, let $\deg_{x_i}(M)$ denote the degree of the variable $x_i$ in $M$. The \emph{total degree} of the monomial $M$ is the sum $\sum_{i=1}^{n}\deg_{x_i}(M)$. In a graph polynomial each monomial has the same total degree equal to the number of edges of $G$. Recall that the \emph{degree} of a polynomial is the maximum of total degrees of its non-vanishing monomials.

We will use the following famous theorem of Alon~\cite{AlonCN}.

\begin{theorem}[Combinatorial Nullstellensatz,~\cite{AlonCN}] Let $P$ be a polynomial in $\mathbb F[x_1,x_2,\dots,x_n]$, where $\mathbb F$ is an arbitrary field of coefficients. Suppose that there is a non-vanishing monomial $x_1^{k_1} x_2^{k_2}\cdots x_n^{k_n}$ in $P$ whose total degree is equal to the degree of $P$. Then, for arbitrary sets $A_i\subseteq \mathbb F$, with $|A_i|=k_i+1$, there exist elements $a_i\in A_i$ such that $$P(a_1,a_2,\dots,a_n)\neq 0\textrm{.}$$
\end{theorem}

In view of this theorem it is convenient to denote by $\A_{\mathbb{F}}(P)$ the least integer $k$ such that the polynomial $P$ has a non-vanishing monomial $M$ whose degree is equal to the degree of $P$ and satisfying $\deg_{x_i}(M)\leqslant k$, for each $i=1,2,\dots,n$.

Let $D_1$, $D_2$ be any two orientations of the graph $G$ sharing the same in-degree sequence.
Observe that the set of edges oriented differently in $D_1$ than in $D_2$ is an Eulerian subgraph of both $D_1$ and $D_2$.
Thus, every monomial corresponding to an acyclic orientation of $G$ is of multiplicity exactly $1$ in $\mathcal{M}_G$.
Such monomials are non-vanishing in $P_G$, and $\A_{\mathbb{F}}(P_G)$ is well defined for every graph $G$.

For the decorated graph polynomial $D_G$, which denotes a collection of polynomials, we define $\A_{\mathbb{F}}(D_G)$ as the least number $k$ such that $\A_{\mathbb{F}}(P)\leqslant k$ holds for every $P$ in $D_G$. It is not hard to demonstrate that for every graph $G$ we have $$\A_{\mathbb{F}}(P_G)\leqslant \A_{\mathbb{F}}(D_G)\leqslant \col (G)-1\textrm{,}$$ where $\col (G)$ is the \emph{coloring number} of $G$, defined as the least $k$ such that the vertices of $G$ can be linearly ordered so that each vertex $v$ has at most $k-1$ neighbors that precede $v$ in the ordering. 

\subsection{Field coloring and graph polynomials}
Let $\mathbb{F}$ be any field and let $\ell$ be any edge labeling of a graph $G$ by the elements of $\mathbb{F}$. Let $D_G$ be a decorated graph polynomial with some fixed decoration over $\mathbb{F}$. Consider the polynomial $D_{G,\ell}$ defined as:

\begin{equation}\label{Graph Polynomial Labels}
D_{G,\ell}(x_1,x_2,\dots,x_n)=\prod_{e=x_ix_j\in E(G)}(a_ex_i+b_ex_j+\ell(e)).
\end{equation}

Our basic observation is formulated as follows.

\begin{theorem}\label{Theorem General Field} Let $G$ be any simple graph, with a graph polynomial $P_G$, and let $\mathbb{F}$ be an arbitrary field. Let $\ell$ be any edge labeling of $G$ by the elements of $\mathbb{F}$. Then $\A_{\mathbb{F}}(D_{G,\ell})=\A _{\mathbb{F}}(D_{G})$.
\end{theorem}
\begin{proof}First notice that the polynomial $D_{G,\ell}$ can be written as:
\begin{equation}\label{Graph Polynomial Labels Sum}
D_{G,\ell}=D_{G}+Q,
\end{equation}
where $Q$ is a polynomial of degree strictly smaller than the degree of $D_G$. Indeed, we obtain the summand $D_G$ by choosing the whole expression $(a_ex_i+b_ex_j)$ in every factor of $D_{G,\ell}$. In every other case, the formed monomial uses at least one constant $\ell(e)$, hence the total degree of such monomial is strictly smaller than the number of edges of $G$, which is equal to the degree of $D_G$. This means that every non-vanishing monomial of $D_G$ does not vanish in $D_{G,\ell}$. This completes the proof. 
    \end{proof}

\subsection{Field coloring of planar graphs}
In~\cite{ZhuAT} Zhu proved that every planar graph $G$ satisfies $\A_{\mathbb{Q}}(P_G)\leqslant 4$, though the same proof works for an arbitrary field. This constitutes an algebraic analog of the famous result of Thomassen~\cite{Thomassen 5} on $5$-choosability of planar graphs. We will derive below a slightly stronger statement.

\begin{theorem}\label{Theorem Zhu Decorated} Let $G$ be a planar graph and let $D_G$ be its decorated graph polynomial over an arbitrary field $\mathbb{F}$. Then $\A_{\mathbb{F}}(D_G)\leqslant 4$.
\end{theorem}

By Theorems~\ref{Theorem General Field} and~\ref{Theorem Zhu Decorated} we get immediately the following result.

\begin{corollary}\label{Corollary Z_5 Planar}
Every planar graph is $\mathbb{F}$-$5$-choosable, where $\mathbb{F}$ is an arbitrary field with at least $5$ elements.
\end{corollary}

The original proof from~\cite{ZhuAT} is by induction with the same scenario as in Thomassen's famous proof from~\cite{Thomassen 5}, except for one unexpected twist. We will give below a purely algebraic proof along similar lines of the following more general statement, stressing the fact that it works for decorated polynomials over an arbitrary field (which is crucial for our applications).

\begin{theorem}\label{Theorem Near-Triangulation}
    Let $G$ be a near-triangulation and let $e=xy$ be an arbitrary edge of the boundary cycle of $G$. Then a decorated graph polynomial $D_{G-e}$ over an arbitrary field $\Bbb F$ contains a non-vanishing monomial $M$ satisfying the following conditions:
    \item[(i)] $\deg_x(M)=\deg_y(M)=0$,
    \item[(ii)] $\deg_v(M)\leqslant 2$, for every boundary vertex $v$,
    \item[(iii)] $\deg_u(M)\leqslant 4$, for every interior vertex $u$.
\end{theorem}

Before we present the proof, let us comment on how Theorem~\ref{Theorem Zhu Decorated} is derived from Theorem~\ref{Theorem Near-Triangulation}.
We can choose any planar drawing of a planar graph $G$, and any edge $e=xy$ of the boundary cycle of the drawing.
Let $M$ be the non-vanishing monomial in $D_{G-e}$ given by Theorem~\ref{Theorem Near-Triangulation}, and observe that the monomial $Mx$ does not vanish in $D_G$ and certifies that $\A_{\mathbb{F}}(D_G) \leq 4$.

\begin{proof}[Proof of Theorem~\ref{Theorem Near-Triangulation}]
    Let us denote $P=D_{G-e}$, and call a non-vanishing monomial $M$ in $P$ satisfying conditions (i)-(iii), a \emph{nice} monomial for $(G,e)$. We use induction on the number of vertices of $G$. It is easy to check that $G=K_3$, a complete graph on vertex set $\{x,y,v\}$, satisfies the assertion. In this case we have $D_{G-e}=(a_1x+b_1v)(a_2y+b_2v)$, and the only nice monomial is $M=v^2$, whose coefficient is $b_1b_2\neq 0$. Hence, $M$ is non-vanishing.
    
    We distinguish two cases.
    \begin{case}\label{Case Chord}
        The boundary cycle of $G$ has a chord.
    \end{case}
    Suppose first that $G$ has a chord $f=wz$. This chord splits $G$ into two subgraphs $G_1$ and $G_2$. We assume that $e$ is in $G_1$, while $f$ belongs to both subgraphs. By induction we assume that both graphs contain nice non-vanishing monomials $M_1$ and $M_2$ for $(G_1,e)$ and $(G_2,f)$, respectively.
    
    Let us denote $P_1=D_{G_1-e}$, and $P_2=D_{G_2-f}$. Then we have $P=P_1P_2$, and we see that the monomial $M=M_1M_2$ appears in the expansion of $P$. We claim that $M$ is a nice monomial for $(G,e)$. It is easy to see that $M$ satisfies conditions (i)-(iii). To see that it is non-vanishing, notice that $M=M_1M_2$ is the only way of expressing the monomial $M$ as a product of two monomials from $P_1$ and $P_2$, respectively. This is because the only common variables of $P_1$ and $P_2$ are $w$ and $z$, and they do not occur in $M_2$. This shows that $$c_M(P)=c_{M_1}(P_1)\cdot c_{M_2}(P_2)\neq 0\textrm{,}$$ confirming that $M$ is non-vanishing in the polynomial $P$.
    
    \begin{case}\label{Case No Chord}
        The boundary cycle of $G$ has no chord.
    \end{case}
    
    Suppose that there is no chord in $G$. Let $v\neq x$ be the neighbor of $y$ on the boundary of $G$. Let $t$ be the other neighbor of $v$ on the outer face, and let $x_1,x_2,\dots, x_k$ be the neighbors of $v$ lying in the interior of $G$. Let $G'=G-v$.
    
    By the inductive assumption, there is a nice monomial $M'$ for $(G',e)$ in the graph polynomial $P'=D_{G'-e}$. So, we have $c_{M'}(P')\neq 0$.
    
    \begin{case2}\label{Subcase Boundary Triangle}
        The boundary face is a triangle.
    \end{case2}
    
    Suppose first that $t=x$, which means that the boundary face is a triangle. In this case the nice monomial $M'$ has the form $$M'=Yx_1^{r_1}x_2^{r_2}\dots x_k^{r_k}\textrm{,}$$ where $r_i\leqslant 2$ for each $i=1,2,\dots,k$, and $Y$ is a monomial consisting of the rest of the variables. Since $M'$ is nice for $(G',e)$, the monomial $Y$ does not contain variables $x$ and $y$, and each other variable $z$ in $Y$ satisfies $\deg_z(Y)\leqslant 4$.

    First notice that $P=P'Q$, where $$Q=(av+bx)(cv+dy)(a_1v+b_1x_1)(a_2v+b_2x_2)\dots(a_kv+b_kx_k)\textrm{.}$$ In the expansion of $Q$ we get the monomial $$N=v^2x_1\dots x_k\textrm{,}$$ whose coefficient is $c_N(Q)=acb_1\cdots b_k\neq0$. Hence, in the expansion of the product $P'Q$ we get the monomial $M=M'N$, which can be written as $$M=Yx_1^{r_1+1}x_2^{r_2+1}\dots x_k^{r_k+1}v^2\textrm{.}$$ Clearly $M$ satisfies conditions (i)-(iii). We claim that it is also non-vanishing in $P$. More specifically, we claim that there is only one way of expressing $M$ as a product $M=AB$ of two monomials, with $A$ from $P'$ and $B$ from $Q$. Indeed, to get $v^2$ in $B$ we have to use the first two factors of $Q$, since otherwise we have $x$ or $y$ in $M$. This implies that $B=N$ and $A=M'$. Thus $$c_M(P)=c_{M'}(P')\cdot c_N(Q)\neq 0\textrm{,}$$ which shows that $M$ is non-vanishing in $P$.
   
    For the remaining subcases, we assume that $t\neq x$.

    \begin{case2}\label{Subcase Special Monomial}
        There is a special monomial.
    \end{case2}
    
    Suppose first that there exists a non-vanishing \emph{special} monomial $S$ in $P'$ which satisfies all conditions (i)-(iii), except that $\deg_t(S)\leqslant 1$ and $\deg_{x_i}(S)\leqslant 3$ for at most one $i$. We may assume without loss of generality that $i=1$. So, we assume that $c_{S}(P')\neq 0$.
    
    This special monomial $S$ can be written in the form $$S=Zx_1^{s_1}x_2^{s_2}\dots x_k^{s_k}t^s\textrm{,}$$ where $s\leqslant 1$, $s_1\leqslant 3$, $s_i\leqslant 2$ for $i=2,3,\dots k$, and $Z$ is some monomial consisting of the rest of the variables. As before we may write the polynomial $P=D_{G-e}$ as the product $P=P'Q$ with $$Q=(av+by)(cv+dt)(a_1v+b_1x_1)\dots(a_kv+b_kx_k)\textrm{.}$$ In the expansion of $Q$ we get the monomial $$N'=vtx_1\dots x_k\textrm{,}$$ with coefficient $c_{N'}(Q)=adb_1\cdots b_k\neq0$. Hence, in the expansion of the product $P'Q$ we get the monomial $M=SN'$, which can be written as $$M=Zx_1^{s_1+1}x_2^{s_2+1}\dots x_k^{s_k+1}t^{s+1}v\textrm{.}$$
    
    Clearly, $M$ satisfies conditions (i)-(iii). Also, as in the previous case, the splitting $M=SN'$ is unique. Indeed, assume that $M=AB$ is any decomposition of $M$ into the product of monomials from $P'$ and $Q$, respectively. To get $v$ in $B$ we have to use the first factor of $Q$, since otherwise we have $y$ in $M$. This already implies that $B=N'$ and $A=S$. Hence, $$c_M(P)=c_{S}(P')\cdot c_{N'}(Q)\neq 0\textrm{,}$$so, $M$ is non-vanishing in $P$.

    \begin{case2}\label{Subcase No Special Monomial}
        There is no special monomial.
    \end{case2}
    Finally, assume that there is no special monomial in $P'$. However, by inductive assumption, there is still a nice non-vanishing monomial $M'$ in the graph polynomial $P'=D_{G'-e}$. This monomial can be written now as $$M'=Yx_1^{r_1}x_2^{r_2}\dots x_k^{r_k}t^r\textrm{,}$$ where $r\leqslant 2$, $r_i\leqslant 2$, for all $i=1,2,\dots,k$, and $Y$ is a monomial consisting of the rest of the variables.
    
    As in Subcase~\ref{Subcase Boundary Triangle}, we have $P=P'Q$, where $$Q=(av+by)(cv+dt)(a_1v+b_1x_1)(a_2v+b_2x_2)\dots(a_kv+b_kx_k)\textrm{.}$$ In the expansion of $Q$ we get the monomial $$N=v^2x_1\dots x_k\textrm{,}$$ with $c_N(Q)=acb_1\cdots b_k\neq0$. Hence, in the expansion of the product $P'Q$ we get the monomial $M=M'N$, which can be written as $$M=Yx_1^{r_1+1}x_2^{r_2+1}\dots x_k^{r_k+1}t^rv^2\textrm{.}$$ Clearly $M$ satisfies conditions (i)-(iii). We claim that there is only one way of expressing $M$ as a product of two monomials, $M=AB$, with $A$ from $P'$ and $B$ from $Q$. Indeed, to get $v^2$ in $B$ we have to choose variable $v$ exactly twice from the factors of $Q$. The first choice must be from the first factor, otherwise $y$ appears in $M$. The second choice must be from the second factor, since otherwise the variable $t$ appears in $B$, while some $x_i$ is missing. Then, in order to get $M=AB$, we would have to have $t^{r-1}$ and $x_i^{r_i+1}$ in the monomial $A$. But then $A$ is a special monomial in $P'$, contrary to our assumption. Hence, we must have $B=N$ and $A=M'$. Thus $$c_M(P)=c_{M'}(P')\cdot c_N(Q)\neq 0\textrm{,}$$ which demonstrates that $M$ is non-vanishing in $P$.
    
    The proof is complete.
    \end{proof}

In~\cite{GrytczukZhu} Grytczuk and Zhu proved that every planar graph $G$ contains a matching $S$ such that $\A_{\mathbb{F}}(P_{G-S})\leqslant 3$. The proof is similar to the above and can be easily modified to give the following result.

\begin{theorem}\label{Theorem Planar Matching} Every planar graph $G$ contains a matching $S$ such that $\A_{\mathbb{F}}(D_{G-S})\leqslant 3$, for an arbitrary field $\mathbb{F}$. Thus, $G-S$ is $\mathbb{F}$-$4$-choosable, and in particular, $\mathbb{Z}_2\times \mathbb{Z}_2$-colorable.
    
\end{theorem}
\subsection{$K_5$-minor-free graphs}
To extend the above results to graphs without a $K_5$-minor we will use the well-known characterization theorem of Wagner~\cite{Wagner}. A similar approach was taken by Abe, Kim, and Ozeki~\cite{AbeKimOzeki} in an extension of the result of Zhu~\cite{ZhuAT} to graphs with no $K_5$-minor.

Recall that a \emph{$k$-clique-sum} of two graphs is a new graph obtained by gluing the two graphs along a clique of size $k$ in each of them, and possibly deleting some edges of the clique. Recall also that the \emph{Wagner graph} $V_8$ is the graph obtained from the cycle $C_8$ by adding four edges joining antipodal pairs of vertices.

\begin{theorem}[Wagner,~\cite{Wagner}] \label{Theorem Wagner}Every edge-maximal graph without a $K_5$-minor can be built recursively from planar triangulations and the graph $V_8$ by clique-sums with cliques on at most $3$ vertices.
\end{theorem}

We need the following result.

\begin{theorem}\label{Theorem Triangulation}Let $G$ be a plane triangulation and let $T$ be any triangle in $G$. Then the decorated graph polynomial $D_{G-E(T)}$ over an arbitrary field $\mathbb{F}$ contains a non-vanishing monomial $N$ such that $\deg_u(N)=0$, for every vertex $u\in V(T)$, and $\deg_w(N)\leqslant 4$, for all other vertices.
\end{theorem}

\begin{proof}Let $V(T)=\{x,y,v\}$ and denote $e=xy$. Suppose first that $T$ is a facial triangle. We may assume that $T$ is the outer face of $G$. By Theorem~\ref{Theorem Near-Triangulation} we know that there is a non-vanishing monomial $M$ in $D_{G-e}$ such that $\deg_x(M)=\deg_y(M)=0$, $\deg_v(M)\leqslant 2$, and $\deg_w(M)\leqslant 4$, for all other variables. In order to form this monomial we have to pick $v$ exactly twice; once from each factor, $(ax+bv)$ and $(cy+dv)$, since we can choose neither $x$, nor $y$. Thus, when we delete the corresponding two edges $xv$ and $yv$ from the graph $G-e$, we must have a non-vanishing monomial $N=M/v^2$ in the polynomial $D_{G-E(T)}$.
    
    If $T$ is not a facial triangle, then we may split the triangulation $G$ into two sub-triangulations, $G_1$ and $G_2$, lying inside and outside the triangle $T$, respectively. Then we may apply the same argument to each sub-triangulation separately to get the desired monomials $N_1$ and $N_2$ in polynomials $D_{G_1-E(T)}$ and $D_{G_2-E(T)}$, respectively. Clearly, we have $$D_{G-E(T)}=D_{G_1-E(T)}D_{G_2-E(T)}\textrm{,}$$and it is easy to see that $N=N_1N_2$ is a non-vanishing monomial in $D_{G-E(T)}$ satisfying the assertion of the theorem. 
\end{proof}

Now we may give the proof of the aforementioned extension of Theorem~\ref{Theorem Zhu Decorated}.

\begin{theorem}\label{Theorem K_5 Polynomial} Let $G$ be a graph without a $K_5$-minor and let $D_G$ be its decorated graph polynomial over an arbitrary field $\mathbb{F}$. Then $\A_{\mathbb{F}}(D_G)\leqslant 4$.
\end{theorem}
\begin{proof} Let $G$ be an edge-maximal $K_5$-minor-free graph. We proceed by induction on the number of terms in a clique-sum giving $G$. So, suppose that $G$ is $k$-clique-sum, $k\leqslant 3$, of two graphs $H$ and $F$, where $H$ is a clique-sum with a smaller number of terms, while $F$ is a plane triangulation or $F=V_8$. Assume by induction that $\A_{\mathbb{F}}(D_{H})\leqslant 4$, and let $M$ be the monomial witnessing this inequality with coefficient $c_M(D_H)\neq 0$.
    
    In the triangulation case, let $\{x,y,z\}$ be the three vertices of the common triangle $T$ in $H$ and $F$. Let $N$ be a monomial in $D_{F-E(T)}$ guaranteed by Theorem~\ref{Theorem Triangulation} with coefficient $c_N(D_{F-E(T)})\neq 0$. We claim that the monomial $MN$ occurs in the polynomial $D_G$ with coefficient $$c_{MN}(D_G)=c_M(D_H)\cdot c_N(D_{F-E(T)})\textrm{.}$$ Indeed, we have an obvious equality $D_G=D_{H}D_{F-E(T)}$ and the only common variables of the two polynomial factors are $x,y,z$, none of which appears in the monomial $N$.
    
    If $F=V_8$, then the reasoning is similar. Notice that $V_8$ is triangle-free, so the clique-sum can be made on one vertex or one edge. Suppose it is the latter situation (the former is even easier). Let $x,y$ be the two common vertices of $H$ and $F$. It is enough to notice that for every edge $e=xy$ of $V_8$ there is an acyclic orientation of $V_8-e$ with in-degrees of both vertices $x$ and $y$ equal to $0$. The monomial $J$ corresponding to this orientation has a non-zero coefficient, the variables $x$ and $y$ do not occur in $J$, while other variables have degrees at most $3$. Thus, as before we have $$c_{MJ}(D_G)=c_M(D_H)\cdot c_J(D_{F-e})\neq 0\textrm{.}$$ This completes the proof.
\end{proof}

By Theorems~\ref{Theorem K_5 Polynomial} and~\ref{Theorem General Field} we get immediately the following results, whose special case extends $\mathbb{Z}_5$-colorability of planar  graphs.
\begin{corollary}\label{Corollary Z_5 K_5 Minor}
    Let $\mathbb{F}$ be an arbitrary field with at least $5$ elements. Then every graph $G$ without a $K_5$-minor is $\mathbb{F}$-$5$-choosable.
\end{corollary}

\subsection{Extending to cyclic group choosability}

The polynomial method is tied to an ambient field. One way of extending the potential results to cyclic groups is to use cyclic subgroups of multiplicative groups of fields with appropriate order and to express the conditions on the coloring using multiplication instead of addition. By this trick we get the following result.

\begin{theorem}\label{Theorem Cyclic Groups}
Let $\Gamma$ be an arbitrary cyclic group of order at least $5$. Then every $K_5$-minor free graph $G$ is $\Gamma$-$5$-choosable.
\end{theorem}

\begin{proof}
Let $\Gamma$ be a (multiplicative) cyclic group of order $m\geqslant 5$. Let $p$ be a prime number such that $\gcd (m,p)=1$. Consider the field $F=\mathbb{F}_{p^{\phi(m)}}$. Its multiplicative group $F^*$ is the cyclic group of order $p^{\phi(m)}-1$. Since $m$ divides $p^{\phi(m)}-1$ (by Euler's theorem), $\Gamma$ is a subgroup of $F^*$.

Let $\ell$ be a labeling of the edges of a graph $G$ by the elements of $\Gamma$. Denote by $\ell_{ij}=\ell(x_ix_j)$ the label of an edge $x_ix_j$ in $G$. For any fixed orientation $\vec{G}$ of $G$, let $d_i$ be the in-degree of the vertex $x_i$. Consider now the following function
$$
f_G(x_1,x_2,\ldots ,x_n)=\prod_{(x_i,x_j)\in E(\vec{G})} (x_ix_j^{-1}-\ell_{ij})=\frac{1}{\prod_{i=1}^n x_i^{d_i}}\prod_{(x_i,x_j)\in E(\vec{G})}(x_i-\ell_{ij}x_j).
$$
By Theorem~\ref{Theorem K_5 Polynomial}, the polynomial
$$
P_G(x_1,x_2,\ldots ,x_n)=\prod_{(x_i,x_j)\in E(\vec{G})}(x_i-\ell(x_ix_j)x_j)
$$
has a nonvanishing monomial of degree at most four. It follows that, for every collection of sets $A_1,A_2,\ldots ,A_n\subseteq \Gamma \subseteq F^*$, each with cardinality at least five, there is a point $(a_1,a_2,\ldots ,a_n)$ in $A_1\times A_2\times \cdots \times A_n$ where $P_G$ is not vanishing, which implies that $f_G(a_1,a_2,\ldots ,a_n)\neq 0$. It follows that $(a_1,a_2,\ldots ,a_n)$ is a $\Gamma$-coloring of $G$ for the labeling $\ell$.

\end{proof}

\subsection{The number of colorings}
In this section we prove the second part of Theorem~\ref{Theorem K_5 Minor Main}. Our main tool is the following general result of Alon and F\"{u}redi~\cite{AlonFuredi}. 

\begin{theorem}[Alon and F\"{u}redi,~\cite{AlonFuredi}]\label{Theorem Alon-Furedi} Let $\mathbb{F}$ be an arbitrary field, let $A_1,A_2,\dots,A_n$ be any non-empty subsets of $\mathbb{F}$, and let $B=A_1\times A_2\times\cdots \times A_n$. Suppose that $P(x_1,x_2,\dots,x_n)$ is a polynomial over $\mathbb{F}$ that does not vanish on all of $B$. Then, the number of points in $B$ for which $P$ has a non-zero value is at least $\min \prod_{i=1}^{n} q_i$, where the minimum is taken over all integers $q_i$ such that $1\leqslant q_i\leqslant |A_i|$ and $\sum_{i=1}^{n} q_i\geqslant \sum_{i=1}^{n}|A_i|-\deg P$.
\end{theorem}

For a convenient use of this result, and for the sake of completeness, we will prove a slightly weaker statement by an argument resembling a beautiful proof of Combinatorial Nullstellensatz, due to Micha\l ek~\cite{Michalek}. A similar approach was taken by Bishnoi, Clark, Potukuchi, and Schmitt~\cite{BishnoiCPS} to get some generalization of the Alon-F\"{u}redi theorem.

We need a simple technical lemma.

\begin{lemma}\label{Lemma Technical}
    Let $a_1,a_2,\dots,a_n$ be positive integers, with $\max a_i=t\geqslant 2$ and $\sum_{i=1}^{n}a_i=S$. Then
    \begin{equation}\label{Inequality Lemma}
    A=\prod_{i=1}^{n}a_i \geqslant t^{\frac{S-n}{t-1}}.
    \end{equation}\end{lemma}
\begin{proof}The proof is by induction on $n$. For $n=1$ we have $A=a_1=t$ and $S=a_1=t$, hence we get equality in (\ref{Inequality Lemma}). For $n\geqslant 2$, let $a_{i_0}=\min a_i=m$. Then, by the inductive assumption, we have
    $$\frac{A}{a_{i_0}}=\frac{A}{m}\geqslant t^{\frac{(S-m)-(n-1)}{t-1}}\textrm{.}$$
    Observe that for every $x\in [1,t]$ we have
    \begin{equation}\label{Lemma Convexity}
    x\geqslant t^{\frac{x-1}{t-1}}.
    \end{equation}
    Indeed, it is not hard to check that the function $f(x)=t^{(x-1)/(t-1)}$ is convex, with $f(1)=1$ and $f(t)=t$. Hence, taking $x=m$, we may write
    $$A\geqslant m\cdot t^{\frac{(S-m)-(n-1)}{t-1}}\geqslant t^{\frac{m-1}{t-1}} \cdot t^{\frac{(S-m)-(n-1)}{t-1}}=t^{\frac{S-n}{t-1}}\textrm{,}$$
    as asserted.
\end{proof}

\begin{theorem}\label{Theorem Alon-Furedi Weak}
Let $\mathbb{F}$ be an arbitrary field, and let $A_1,A_2,\dots,A_n$ be any non-empty subsets of $\mathbb{F}$, with $S=\sum_{i=1}^{n}|A_i|$ and $t=\max |A_i|$. Let $B=A_1\times A_2\times\cdots \times A_n$ and suppose that $P(x_1,x_2,\dots,x_n)$ is a polynomial over $\mathbb{F}$ of degree $\deg P=d$, that does not vanish on all of $B$. Then, the number of points in $B$ for which $P$ has a non-zero value is at least $$t^{\frac{S-n-d}{t-1}}\textrm{,}$$ provided that $S\geqslant n+d$ and $t\geqslant 2$.
\end{theorem}

\begin{proof}The proof is by induction on $d$ and $n$. If $d=0$, then $P$ equals some non-zero constant $c\in \mathbb{F}$, and therefore all points of $B$ are non-vanishing for $P$. There are exactly $\prod_{i=1}^{n}|A_i|$ of them, so, the assertion follows from Lemma~\ref{Lemma Technical} by putting $a_i=|A_i|$.
    
    For $n=1$ and arbitrary $d\geqslant 1$, we know that the number of roots of a polynomial $P$ is at most $d$. Hence, the number of elements of $A_1$ for which $P$ is non-zero is at least $t-d$. So, by the assumption that $t\geqslant d+1$ and the inequality (\ref{Lemma Convexity}), we have $$t-d\geqslant t^{\frac{t-1-d}{t-1}}=t^{\frac{S-1-d}{t-1}}\textrm{,}$$as $S=t$ in this case.
    
    Assume now that $d\geqslant 1$ and $n\geqslant 2$. Let $|A_1|=j$, and assume, without loss of generality, that there is another set $A_i$, with $|A_i|=t$. Let $b\in A_1$ be any element, and let us divide the polynomial $P$ by $(x_1-b)$:
    $$P=(x_1-b)Q_b+R_b\textrm{.}$$
    Observe that $\deg Q_b=\deg P-1=d-1$, $\deg R_b\leqslant d$, and that the polynomial $R_b$ does not contain the variable $x_1$.
    
    Suppose first that the polynomial $R_b$ vanishes at all points of the grid $A_2\times \cdots \times A_n$. This implies that $j\geqslant 2$, since otherwise, the polynomial $P$ would vanish over the whole grid $B$, contrary to the assumption. Furthermore, each non-vanishing point of $P$ in $B$ is at the same time a non-vanishing point of $Q_b$ in the grid $(A_1-b)\times A_2\times \cdots \times A_n$, and vice versa. Thus, by the inductive assumption we get that $P$ has at least $$t^{\frac{(S-1)-n-(d-1)}{t-1}}=t^{\frac{S-n-d}{t-1}}$$ non-vanishing points in $B$.
    
    Finally, suppose that for every $b\in A_1$, the polynomial $R_b$ has some non-vanishing points in the grid $A_2\times \cdots \times A_n$. Then each such point can be extended to a non-vanishing point of $P$ by setting $x_1=b$. By the inductive assumption on $n$, the number of such points for each $R_b$ is at least $$t^{\frac{(S-j)-(n-1)-d}{t-1}}\textrm{.}$$ Hence, the total number of non-vanishing points of $P$ in the grid $B$ is at least $$j\cdot t^{\frac{(S-j)-(n-1)-d}{t-1}}\geqslant t^{\frac{j-1}{t-1}} \cdot t^{\frac{(S-j)-(n-1)-d}{t-1}}=t^{\frac{S-n-d}{t-1}}\textrm{,}$$by the inequality (\ref{Lemma Convexity}). The proof is complete.
    \end{proof}

The above result and Theorem~\ref{Theorem K_5 Polynomial} easily imply the following corollary.

\begin{corollary}\label{Corollary Number of Colorings} Let $G$ be a graph on $n$ vertices with no $K_5$-minor. Let $\mathbb{F}$ be an arbitrary field with at least $5$ elements. Then for every assignment of lists of size $5$ from $\mathbb{F}$ to the vertices of $G$ and every edge labeling $\ell$, there are at least $5^{n/4}$ different $\mathbb{F}$-colorings of $G$ from these lists.
\end{corollary}

\begin{proof} Let $G$ be a graph satisfying assumptions of the corollary and let $P_G$ be its polynomial. Let $A_i\subseteq \mathbb{F}$ be lists of size $5$ assigned to the vertices of $G$. So, keeping notation from Theorem~\ref{Theorem Alon-Furedi Weak}, we have $t=5$ and $S=5n$. Since the number of edges of $G$ is at most $3n-6$ (by Theorem~\ref{Theorem Wagner}), we have $d=\deg P_G\leqslant 3n-6$. Hence, the number of $\mathbb{F}$-colorings is at least
    $$t^{(S-n-d)/(t-1)}\geqslant 5^{(5n-n-(3n-6))/4}=5^{(n+6)/4}>5^{n/4}\textrm{,}$$
    as asserted.
    \end{proof}

\section{Discussion}
Let us conclude the paper with some further observations and possible extensions of our results.

First, let us point on other improvements of existing results one may easily get by using graph polynomials. For instance, in~\cite{Thomassen Exponentially 5} Thomassen proved that every planar graph on $n$ vertices has at least $2^{n/9}$ colorings from arbitrary lists of size $5$. By a direct application of the Alon-F\"{u}redi theorem one may improve this bound to $5^{n/4}$, arguing in the same way as in the proof of Corollary~\ref{Corollary Number of Colorings}. A similar improvement can be made for planar graphs of girth at least $5$. Thomassen proved that every such graph $G$ with $n$ vertices has at least $2^{n/10000}$ colorings from lists of size $3$, while the Alon-F\"{u}redi theorem gives $3^{n/6}$.

Another striking consequence can be formulated in connection to the famous theorem of Brooks~\cite{Brooks} which asserts that every connected graph $G$ of maximum degree $\Delta$ is $\Delta$-colorable, except for the two cases, when $G$ is a clique or an odd cycle (see~\cite{Zajac} for a short proof of some more general versions). By Theorem~\ref{Theorem Alon-Furedi Weak} we get that the number of colorings in this case is at least $$(\sqrt{\Delta})^{n(1-\frac{1}{\Delta-1})}\textrm{.}$$ By a different argument based on the probabilistic method, this can be improved to roughly $$(\Delta/e)^{n(1-\frac{\log \Delta}{\Delta})}\textrm{,}$$ as noticed by Alon (personal communication).

Another advantage of using graph polynomials is that any upper bound on $\A_{\mathbb{F}}(P_G)$ not only gives an upper bound for the choosability of $G$, but also for its \emph{game} variant, known as \emph{paintability} or \emph{online choosability}. Roughly speaking, a graph $G$ is \emph{$k$-paintable}, if Alice, a color-blind person, can color the vertices of $G$ from any lists of size $k$ in a sequence of rounds; in each round a new color is highlighted in the lists and Alice must decide which vertices will be painted by this color. This idea was introduced independently by Schauz~\cite{Schauz Paint Correct} and Zhu~\cite{Zhu Online}. It is clear that every $k$-paintable graph is $k$-choosable, but not the other way around. In~\cite{Schauz} Schauz proved that every graph $G$ satisfying $\A_{\mathbb{F}}(G)\leqslant k-1$ is $k$-paintable (see~\cite{GrytczukJZ} for a different, purely algebraic proof). Hence, all our results for field choosability transfer to field paintability in a direct way.

Finally, let us mention about the intriguing conjecture relating group coloring and list coloring of graphs, stated by Kr\'{a}l', Pangr\'{a}c, and Voss in~\cite{KralPV}. They suspect that if a graph $G$ is $\Gamma$-colorable for every Abelian group of order at least $k$, then $G$ is $k$-choosable. We offer the following, somewhat provocative, strengthening of this statement.

\begin{conjecture}
If a graph $G$ is $\Gamma$-colorable for every Abelian group $\Gamma$ of order at least $k$, then $\A_{\mathbb{Q}}(P_G)\leqslant k-1$. This would imply that $G$ is $k$-paintable and $k$-choosable.
\end{conjecture}

\end{document}